\documentclass[11pt]{article}
\usepackage[english]{babel}
\usepackage[all]{xy}
\usepackage{fancyhdr}
\usepackage{setspace, amsmath, amsthm, amssymb, amsfonts, amscd, epic, graphicx, ulem, dsfont}
\usepackage{multirow}
\usepackage{amsthm}
\usepackage[T1]{fontenc}
\usepackage{answers,mathptmx}
\newtheorem{theorem}{Theorem}
\newtheorem{definition}[theorem]{Definition}

\newtheorem{proposition}[theorem]{Proposition}
\newtheorem{corollary}[theorem]{Corollary}
\newtheorem{example}[theorem]{Example}

\newtheorem{remark}[theorem]{Remark}

\makeatletter

\@addtoreset{equation}{section}
\makeatother

\begin{document}

\markboth{Aicha Benkartab and Ahmed Mohammed Cherif}
{Geometry of Harmonic Identity Maps}

%

%

\title{Geometry of Harmonic Identity Maps}

\author{AICHA BENKARTAB\\
Department of  Mathematics, Mascara University, Mascara 29000, Algeria\\
aicha\_benkartab@yahoo.com\\\\
AHMED MOHAMMED CHERIF\\
Department of  Mathematics, Mascara University, Mascara 29000, Algeria\\
med\_cherif\_ahmed@yahoo.fr}

\date{}

\maketitle

\begin{abstract}
An identity map $(M,g)\longrightarrow(M,g)$ is a harmonic from a Riemannian manifold $(M,g)$ onto itself.
In this paper, we study the harmonicity of identity maps $(M,g)\longrightarrow(M,g-df\otimes df)$ and
$(M,g-df\otimes df)\longrightarrow(M,g)$ where $f$ is a smooth function with gradient norm $<1$ on $(M,g)$.
We construct new examples of identity harmonic maps. We define a symmetric tensor field on $M$ whose properties are related to the harmonicity of these identity maps.\\\\
\textit{ keywords:} Riemannian geometry; Identity map; Harmonic map.\\
\textit{ Mathematics Subject Classification 2020:} 53C20; 58E20.
\end{abstract}

\section{Introduction}	
Let $(M, g)$ and $(N,h)$ be two Riemannian manifolds. The energy functional of a map $\varphi\in C^\infty(M,N)$ is defined by
\begin{equation}\label{eq1.1}
    E(\varphi)=\frac{1}{2}\int_M|d\varphi|^2 v^g,
\end{equation}
where $|d\varphi|$ is the Hilbert-Schmidt norm of the differential $d\varphi$, and $v^g$ is the volume element on $(M,g)$. A map $\varphi\in C^\infty(M,N)$ is called harmonic if it is a critical point of the energy functional, that is, if it is a solution of the Euler Lagrange equation associated to (\ref{eq1.1})
\begin{equation}\label{eq1.2}
\tau(\varphi)=\operatorname{trace}_g\nabla d\varphi
=\sum_{i=1}^m\left[\nabla^\varphi _{e_i} d\varphi(e_i)-d\varphi(\nabla^M _{e_i}e_i)\right]=0,
\end{equation}
where $\{e_i\}_{i=1}^m$ is an orthonormal local frame on $(M,g)$, $\nabla^{M}$ the Levi-Civita connection of $(M,g)$,
and $\nabla^{\varphi}$ denote the pull-back connection on $\varphi^{-1}TN$.
Harmonic maps are solutions of a second order nonlinear elliptic system. One can refer to \cite{EL1,EL2,eells} for background on
harmonic maps.\\
In \cite{ba,ba2}, the authors studied biharmonic maps between
Riemannian manifolds equipped with Riemannian metrics of the form $g+df\otimes df$ for some smooth function $f$,
in particular they gave the condition for the biharmonicity of the identity maps.
In this paper, we deform the domain metric (resp. the codomain metric) of the identity map $(M,g)\longrightarrow(M,g)$
by $\tilde{g}= g-df\otimes df$ where $f \in C^\infty(M)$ such that $\| \operatorname{grad} f\|<1$ in order to render the identity map $(M,\tilde{g})\longrightarrow(M,g)$ (resp. $(M,g)\longrightarrow(M,\tilde{g})$) a harmonic map.
We give new examples and properties of harmonic identity maps.

\section{The harmonicity of the identity maps}
Let $M$  be a smooth manifold equipped with Riemannian metric $g$, and
$f$ a smooth function on $M$ such that the following deformation $\tilde{g}= g- df\otimes df$
is a Riemannian metric on $M$, that is, it is positive definite.
\begin{remark}
 Let $M$ be a connected $m$-dimensional manifold. A continuous function $f: M\longrightarrow \mathbb{R}$ is called proper if $f^{-1}(K) $ is compact for every compact $K$ in $\mathbb{R}$. A Riemannian manifold $(M,g)$ is complete if and only
 if there exists a proper smooth function $f: M\longrightarrow \mathbb{R}$ such that $g - df \otimes df$ is positive definite (see \cite{U1994}).
\end{remark}
In \cite{MM23}, B. Merdji and A. Mohammed Cherif proved that, if
$$\| \operatorname{grad} f\|^{2}=g(\operatorname{grad} f,\operatorname{grad} f)<1,$$
on $M$, where $\operatorname{grad}$ is the gradient operator with respect to $g$, then
$\tilde{g}$ is a Riemannian metric on $M$. Moreover, the Levi-Civita connection of $(M,\tilde{g})$ is given by
\begin{equation}\label{eq2.1}
\widetilde{\nabla}_{X}Y=\nabla_{X}Y-\frac{\operatorname{Hess}_{f}(X,Y)}{1-\| \operatorname{grad} f\|^{2}}\operatorname{grad}f,\quad\forall X,Y\in\Gamma(TM),
\end{equation}
where $\nabla$ is the Levi-Civita connection of $(M,g)$, and $\operatorname{Hess}_{f}$ denote the Hessian of $f$ with respect to $g$.\\
First, we deform the codomain metric of the identity map $I:(M,g) \longrightarrow (M,g)$
by $\tilde{g}= g-df\otimes df$, where $f \in C^\infty(M)$ such that $\|\operatorname{grad}f\|<1$ on $M$. We denote by
\begin{eqnarray*}
  \tilde{I}_c:(M,g) &\longrightarrow& (M,\tilde{g}), \\
  x &\longmapsto& x
\end{eqnarray*}
In the following Theorem, we compute the tension field of $\tilde{I}_c$.
\begin{theorem}\label{th1}
The tension field of the identity map $\tilde{I}_c:(M,g) \longrightarrow (M,\tilde{g})$ is given by
\begin{eqnarray*}
\tau(\tilde{I}_c)
&=&-\frac{\Delta f}{1-\|\operatorname{grad}f\|^{2}}\operatorname{grad}f,
\end{eqnarray*}
where $\Delta f$ is the Laplacian of $f$ with respect to $g$.
\end{theorem}
\begin{proof}
Let $\{e_i\}_{i=1}^m$ be a local orthonormal frame on $(M,g)$, we compute
\begin{eqnarray}\label{eq2.2}
\tau(\tilde{I}_c)
&=&\sum_{i=1}^m\left[\nabla^{\tilde{I}_c}_{e_{i}}d\tilde{I}_c(e_{i})-d\tilde{I}_c(\nabla_{e_{i}}e_{i})\right]\nonumber\\
&=&\sum_{i=1}^m\left[\widetilde{\nabla}_{e_{i}}e_{i}-{\nabla}_{e_{i}}e_{i}\right]\nonumber\\
&=&-\sum_{i=1}^m\frac{\operatorname{Hess}_{f}(e_{i},e_{i})}{1-\|\operatorname{grad}f\|^{2}}\operatorname{grad}f\nonumber\\
&=&-\frac{\Delta f}{1-\|\operatorname{grad}f\|^{2}}\operatorname{grad}f.
\end{eqnarray}
\end{proof}
From Theorem \ref{th1}, we deduce the following Corollaries.
\begin{corollary}
The identity map $\tilde{I}_c$ is harmonic if and only if $f$ is harmonic function on $(M,g)$.
\end{corollary}
\begin{corollary}
Let $\epsilon\in(0,1]$. Then, $\tilde{I}_c:(M,g) \longrightarrow (M,g-df\otimes df)$
is harmonic if and only if the identity map $(M,g) \longrightarrow (M,g-\epsilon^2df\otimes df)$ is harmonic.
\end{corollary}
\begin{corollary}
Let $f_1,f_2 \in C^\infty(M)$ such that $\|\operatorname{grad}f_1\|+\|\operatorname{grad}f_2\|<1$ on $M$.
If the identity maps $(M,g) \longrightarrow (M,g-df_1\otimes df_1)$ and $(M,g) \longrightarrow (M,g-df_2\otimes df_2)$
are harmonic, then the identity map $(M,g) \longrightarrow (M,g-d(f_1+f_2)\otimes d(f_1+f_2))$ is harmonic.
\end{corollary}
\begin{remark}
In Euclidean space (resp. hyperbolic space) for the identity map $\tilde{I}_c$ to be harmonic it is sufficient to choose a harmonic function in this space satisfying $\|\operatorname{grad}f\|<1$. For example,
 \begin{itemize}
   \item if $M=\mathbb{R}^{n}$ equipped with standard Riemannian metric $g=dx_1^2+...+dx_n^2$, the function $f(x_{1},...,x_{n})=a_{1}x_{1}+...+a_{n}x_{n}$ where $a_1,...,a_n\in\mathbb{R}$ is harmonic, and note that $\|\operatorname{grad}f\|<1$ for $a_1^2+...+a_n^2<1$. Here, $\tilde{g}_{ij}=\delta_{ij}-a_ia_j$ for all $i,j=1,...,n$;
   \item if $M=\mathbb{H}^n$ equipped with Riemannian metric $g=x_n^{-2}(dx_1^2+...+dx_n^2)$, the function $f(x_{1},...,x_{n})=a+bx_{n}^{n-1}$ where $a,b\in\mathbb{R}$ is harmonic.
       We have $$\|\operatorname{grad}f\|^2=b^2(n-1)^2x_n^{2(n-1)}.$$
       Hence, $\|\operatorname{grad}f\|<1$ on $M_b=\{x\in \mathbb{R}^n\,|\,0<b(n-1)x_n^{n-1}<1\}$. Thus,
       the identity map $\tilde{I}_c:(M_b,g) \longrightarrow (M_b,\tilde{g})$ is harmonic. Here,
       $$\tilde{g}=x_n^{-2}(dx_1^2+...+dx_{n-1}^2)+x_n^{-2}(1-b^2(n-1)^2x_n^{2(n-1)})dx_{n}^2.$$
 \end{itemize}
But, if $M$ is a compact Riemannian manifold (e.g., the sphere $\mathbb{S}^n$), by using Green's Theorem (see for example \cite{baird}), the identity map $\tilde{I}_c$ from $(M,g)$ into $(M,\tilde{g})$ is harmonic if and only if the function $f $ is constant.
 \end{remark}
Now, we denote by
\begin{eqnarray*}
  \tilde{I}_d:(M,\tilde{g}) &\longrightarrow& (M,g), \\
  x &\longmapsto& x
\end{eqnarray*}
the identity  map, where $\tilde{g}= g-df\otimes df$ for some smooth function $f$ on $M$ such that $\|\operatorname{grad}f\|<1$.
We obtain the following results.
\begin{theorem}\label{th2}
The identity map $\tilde{I}_d$ is harmonic if and only if
\begin{eqnarray*}
   \operatorname{Hess}_{f}(\operatorname{grad}f,\operatorname{grad}f)+(1-\|\operatorname{grad}f\|^{2})\Delta f=0.
\end{eqnarray*}
\end{theorem}
\begin{proof}
Let $\{E_{i}\}^{m}_{i=1}$ be a local orthonormal frame on $(M,g)$ with $E_{1}=\frac{1}{\|\operatorname{grad}f\|}\operatorname{grad}f$.
Thus, $\{\tilde{E}_{i}\}^{m}_{i=1}$ is a local orthonormal frame on $(M,\tilde{g})$ such that $\tilde{E}_{1}=\frac{1}{\sqrt{1-\|\operatorname{grad}f\|^{2}}}E_{1}$ and $\tilde{E}_{i}=E_{i}$ for $i=2,...,m$.
We compute the tension field of $\tilde{I}_d$
\begin{eqnarray}\label{eq2.3}
\tau(\tilde{I}_d)
&=&\nonumber\operatorname{trace}_{\tilde{g}}\nabla d\tilde{I}_d\\
&=&\nonumber\sum^{m}_{i=1}\left[\nabla_{\tilde{E}_{i}}^{\tilde{I}_d}d\tilde{I}_d(\tilde{E}_{i})
-d\tilde{I}_d(\widetilde{\nabla}_{\tilde{E}_{i}}\tilde{E}_{i})\right]\\
&=&\nonumber\nabla_{\tilde{E}_{1}}\tilde{E}_{1}-\widetilde{\nabla}_{\tilde{E}_{1}}\tilde{E}_{1}+
\sum^{m}_{i=2}\left[\nabla_{\tilde{E}_{i}}\tilde{E}_{i}-\widetilde{\nabla}_{\tilde{E}_{i}}\tilde{E}_{i}\right]\\
&=&\nonumber\frac{\operatorname{Hess}_{f}(\tilde{E}_{1},\tilde{E}_{1})}{1-\|\operatorname{grad}f\|^{2}}\operatorname{grad}f+
\sum^{m}_{i=2}\frac{\operatorname{Hess}_{f}(\tilde{E}_{i},\tilde{E}_{i})}{1-\|\operatorname{grad}f\|^{2}}\operatorname{grad}f\\
&=&\frac{\operatorname{Hess}_{f}(\operatorname{grad}f,\operatorname{grad}f)}{(1-\|\operatorname{grad}f\|^{2})^{2}\|\operatorname{grad}f\|^{2}}\operatorname{grad}f
+\sum^{m}_{i=2}\frac{\operatorname{Hess}_{f}(E_{i},E_{i})}{1-\|\operatorname{grad}f\|^{2}}\operatorname{grad}f.\qquad
\end{eqnarray}
Note that
\begin{eqnarray}\label{eq2.4}
  \sum^{m}_{i=2}\frac{\operatorname{Hess}_{f}(E_{i},E_{i})}{1-\|\operatorname{grad}f\|^{2}}\operatorname{grad}f.
   &=&\nonumber \sum^{m}_{i=1}\frac{\operatorname{Hess}_{f}(E_{i},E_{i})}{1-\|\operatorname{grad}f\|^{2}}\operatorname{grad}f
       -\frac{\operatorname{Hess}_{f}(E_{1},E_{1})}{1-\|\operatorname{grad}f\|^{2}}\operatorname{grad}f. \\
   &=&\nonumber \frac{\Delta f}{1-\|\operatorname{grad}f\|^{2}}\operatorname{grad}f\\
   & & -\frac{\operatorname{Hess}_{f}(\operatorname{grad}f,\operatorname{grad}f)}
       {(1-\|\operatorname{grad}f\|^{2})\|\operatorname{grad}f\|^{2}}\operatorname{grad}f.
\end{eqnarray}
By substituting (\ref{eq2.4}) in (\ref{eq2.3}), we get
\begin{eqnarray}\label{eq2.4-2}
\tau(\tilde{I}_d)
&=&\left[\frac{\operatorname{Hess}_{f}(\operatorname{grad}f,\operatorname{grad}f)}{(1-\|\operatorname{grad}f\|^{2})^{2}}+
\frac{\Delta f}{1-\|\operatorname{grad}f\|^{2}}\right]\operatorname{grad}f.
\end{eqnarray}
\end{proof}
From Theorem \ref{th2}, we obtain the following Corollaries.
\begin{corollary}
Let $\epsilon\in(0,1)$. If $\tilde{I}_d:(M,g-df\otimes df) \longrightarrow (M,g)$
is harmonic, then the identity map $(M,g-\epsilon^2df\otimes df) \longrightarrow (M,g)$ is harmonic if and only if the function $f$ is harmonic on $(M,g)$.
\end{corollary}
\begin{corollary}
Let $(M,g)$ be a Riemannian manifold and $f\in C^{\infty}(M)$ an affine function, i.e., $\operatorname{Hess}_{f}=0$. Then, the identity map $\tilde{I}_d$ is harmonic.
\end{corollary}
\begin{corollary}
We assume that $f$ is convex function on $(M,g)$, i.e., $\operatorname{Hess}_{f}\geq0$.
If $\tilde{I}_d$ is harmonic map, then $\tilde{I}_c$ is harmonic map.
\end{corollary}
\begin{corollary}\label{co3}
Let $(M,g)$ be a Riemannian manifold. We assume that $\operatorname{Hess}_{f}=\lambda g$ for some $\lambda \in C^\infty(M)$. Then, $\tilde{I}_d$ is harmonic if and only if $\lambda=0$.
\end{corollary}
\begin{proof}
According to Theorem \ref{th2}, the identity map $\tilde{I}_d$ is harmonic if and only if
\begin{eqnarray*}
\lambda\left(\|\operatorname{grad}f\|^{2}+m(1-\|\operatorname{grad}f\|^{2})\right)=0,
\end{eqnarray*}
it is equivalent to the following
\begin{eqnarray}\label{eq2.5}
\lambda\left(m+(1-m)\|\operatorname{grad}f\|^{2}\right)=0.
\end{eqnarray}
Note that, if $\|\operatorname{grad}f\|^{2}=m/(m-1)$, we get $\|\operatorname{grad}f\|^{2}=1+1/(m-1)>1$.
The Corollary \ref{co3} follows from equation (\ref{eq2.5}).
\end{proof}
\begin{theorem}\label{th3}
Let $(M,g)$ be a compact Riemannian manifold. Then, the identity map $\tilde{I}_d$ is harmonic if and only if the function $f$ is constant on $M$.
\end{theorem}
\begin{proof}
Let $\{E_{i}\}^{m}_{i=1}$ be a local orthonormal frame on $(M,g)$ with $E_{1}=\frac{1}{\|\operatorname{grad}f\|}\operatorname{grad}f$.
So that $\{\tilde{E}_{i}\}^{m}_{i=1}$ is a local orthonormal frame on $(M,\tilde{g})$ with $\tilde{E}_{1}=\frac{1}{\sqrt{1-\|\operatorname{grad}f\|^{2}}}E_{1}$ and $\tilde{E}_{i}=E_{i}$ for $i=2,...,m$.
We compute the Laplacian of $f$ on $(M,\tilde{g})$
\begin{eqnarray}\label{eq2.6}
  \widetilde{\Delta}f
   &=&\nonumber \sum_{i=1}^m \left[\tilde{E}_i(\tilde{E}_i(f))-(\widetilde{\nabla}_{\tilde{E}_i}\tilde{E}_i)(f)\right]\\
   &=&\nonumber  \tilde{E}_1(\tilde{E}_1(f))-(\widetilde{\nabla}_{\tilde{E}_1}\tilde{E}_1)(f)+\sum_{i=2}^m \left[\tilde{E}_i(\tilde{E}_i(f))-(\widetilde{\nabla}_{\tilde{E}_i}\tilde{E}_i)(f)\right]\\
   &=&\nonumber\frac{1}{1-\|\operatorname{grad}f\|^2}\left[E_1(E_1(f))-(\widetilde{\nabla}_{E_1}E_1)(f)\right]\\
   & &+\sum_{i=2}^m \left[E_i(E_i(f))-(\widetilde{\nabla}_{E_i}E_i)(f)\right].
\end{eqnarray}
By using (\ref{eq2.1}) and (\ref{eq2.6}), we find that
\begin{eqnarray}\label{eq2.7}
  \widetilde{\Delta}f
   &=&\nonumber\frac{1}{1-\|\operatorname{grad}f\|^2}\left[E_1(E_1(f))-(\nabla_{E_1}E_1)(f)
      +\frac{\operatorname{Hess}_f(E_1,E_1)\|\operatorname{grad}f\|^2}{1-\|\operatorname{grad}f\|^2}\right]\\
   & &\nonumber+\sum_{i=2}^m \left[E_i(E_i(f))
   -(\nabla_{E_i}E_i)(f)+\frac{\operatorname{Hess}_f(E_i,E_i)\|\operatorname{grad}f\|^2}{1-\|\operatorname{grad}f\|^2}\right]\\
   &=&\nonumber  \frac{\|\operatorname{grad}f\|^2}{1-\|\operatorname{grad}f\|^2}\left[E_1(E_1(f))-(\nabla_{E_1}E_1)(f)\right]
        +\frac{\operatorname{Hess}_f(E_1,E_1)\|\operatorname{grad}f\|^4}{(1-\|\operatorname{grad}f\|^2)^2}\\
   & &  +\frac{\Delta f}{1-\|\operatorname{grad}f\|^2}.
\end{eqnarray}
From (\ref{eq2.7}) and the definition of the Laplacian on $(M,g)$, we get
\begin{eqnarray}\label{eq2.8}
  \widetilde{\Delta}f
   &=&\nonumber  \frac{\operatorname{Hess}_f(\operatorname{grad}f,\operatorname{grad}f)}{1-\|\operatorname{grad}f\|^2}
        +\frac{\operatorname{Hess}_f(\operatorname{grad}f,\operatorname{grad}f)\|\operatorname{grad}f\|^2}{(1-\|\operatorname{grad}f\|^2)^2}
        +\frac{\Delta f}{1-\|\operatorname{grad}f\|^2}\\
   &=&  \frac{\operatorname{Hess}_f(\operatorname{grad}f,\operatorname{grad}f)}{(1-\|\operatorname{grad}f\|^2)^2}
        +\frac{\Delta f}{1-\|\operatorname{grad}f\|^2}.
\end{eqnarray}
From (\ref{eq2.4-2}) and (\ref{eq2.8}), we conclude that
$\tau(\tilde{I}_d)=(\widetilde{\Delta}f)\operatorname{grad}f$. The proof follows from Green's theorem.
\end{proof}
\begin{remark}
The identity map $\tilde{I}_d$ is harmonic if and only if the function $f$ is harmonic with respect to Riemannian metric $\tilde{g}$.
\end{remark}
\begin{example}
Let $f(x_1,...,x_n)=a_1x_1+...+a_nx_n$ for all $(x_1,...,x_n)\in\mathbb{R}^n$, where $a_1,...,a_n\in \mathbb{R}$ such that
$a_1^2+...+a_n^2<1$. Then, the identity map $\tilde{I}_d$ is harmonic from $(\mathbb{R}^n,\tilde{g})$ into $(\mathbb{R}^n,g)$ where $g=dx_1^2+...+dx_n^2$ and here $$\tilde{g}=\sum_{i=1}^n(1-a_i^2)dx_i^2-\sum_{i,j=1}^na_ia_jdx_idx_j.$$
\end{example}
\begin{example}
Let $f(x_1,...,x_n)=a+b\,(x_1+...+x_{n-1})$ for all $(x_1,...,x_n)\in M$, where $a,b\in \mathbb{R}$ $(b\neq0)$ and
$M=\{(x_1,...,x_n)\in\mathbb{R}^n\,|\,0<(n-1)b^2x_n^2<1\}\subset\mathbb{H}^n$. Then, the identity map $\tilde{I}_d$ is harmonic from $(M,\tilde{g})$ into $(M,g)$ where $g=x_n^{-2}(dx_1^2+...+dx_n^2)$. Here, $\|\operatorname{grad}f\|^2=(n-1)b^2x_n^2$ and the Riemannian metric $\tilde{g}$ is given by  $$\tilde{g}=x_n^{-2}(1-b^2x_n^2)\sum_{i=1}^{n-1}dx_i^2+x_n^{-2}dx_n^2-b^2\sum_{i,j=1}^{n-1}dx_idx_j.$$
\end{example}
\section{Characterization of harmonic identity maps}
\begin{definition}
Let $(M,g)$ be a Riemannian manifold and $f\in C^\infty(M)$. We define a symmetric tensor field on $M$ by
\begin{eqnarray*}
   \chi=\|\operatorname{grad}f\|^{2}\operatorname{Hess}_{f}+(\Delta f)(g-df\otimes df).
\end{eqnarray*}
\end{definition}
\begin{remark}
We assume that $\|\operatorname{grad}f\|<1$ on $M$. We have
\begin{itemize}
  \item the identity map $\tilde{I}_c:(M,g) \longrightarrow (M,g-df\otimes df)$ is harmonic if and only if
  $$\operatorname{trace}_{g}\chi=0;$$
  \item the identity map $\tilde{I}_d:(M,g-df\otimes df) \longrightarrow (M,g)$ is harmonic if and only if
  $$\chi(\operatorname{grad}f,\operatorname{grad}f)=0;$$
  \item if $M$ is a compact manifold, then $\tilde{I}_c$ and $\tilde{I}_d$ are harmonic if and only if $\chi=0$.
\end{itemize}
\end{remark}
In the following Theorem we compute the divergence of $\chi$ this is in order to find other properties about this symmetric tensor field.
\begin{theorem}\label{th3}
Let $(M,g)$ be a Riemannian manifold and $f\in C^\infty(M)$. The divergence of $\chi$ is given by
\begin{eqnarray*}
  (\operatorname{div}\chi)(V)
   &=& g(\nabla_{V}\operatorname{grad}f,\operatorname{grad}\|\operatorname{grad}f\|^{2})
                +\|\operatorname{grad}f\|^{2}\operatorname{Ric}(\operatorname{grad}f,V)\\
   & & +(1+\|\operatorname{grad}f\|^{2})g(\operatorname{grad}\Delta f,V)-g(\operatorname{grad}\Delta f,\operatorname{grad}f)g(V,\operatorname{grad}f)\\
   & & -(\Delta f)^2g(V,\operatorname{grad}f)
                -\frac{1}{2}(\Delta f)g(V,\operatorname{grad}\|\operatorname{grad}f\|^{2}),
\end{eqnarray*}
for all $V\in\Gamma(TM)$.
\end{theorem}
\begin{proof}
Let $\{e_i\}_{i=1}^m$ be a local orthonormal frame on $(M,g)$. We suppose that $\nabla_{e_i}e_j=0$ at $x\in M$ for all $i,j=1,...,m$. For all $j=1,...,m$, we have at $x$
\begin{eqnarray}\label{eq3.1}
  (\operatorname{div}\chi)(e_j)
   &=&\nonumber \sum_{i=1}^m(\nabla_{e_i}\chi)(e_i,e_j) \\
   &=&\nonumber \sum_{i=1}^m e_i\left[\chi(e_i,e_j)\right]\\
   &=&\nonumber \sum_{i=1}^m e_i\Big[\|\operatorname{grad}f\|^{2}g(\nabla_{e_i}\operatorname{grad}f,e_j)+(\Delta f)g(e_i,e_j)\\
   & &          -(\Delta f)g(e_i,\operatorname{grad}f)g(e_j,\operatorname{grad}f)\Big].
\end{eqnarray}
By direct calculation we obtain from (\ref{eq3.1}) the following
\begin{eqnarray}\label{eq3.2}
  (\operatorname{div}\chi)(e_j)
   &=&\nonumber g(\nabla_{e_j}\operatorname{grad}f,\operatorname{grad}\|\operatorname{grad}f\|^{2})
                +\|\operatorname{grad}f\|^{2}g(\operatorname{trace}_g(\nabla)^2\operatorname{grad}f,e_j)\\
   & &\nonumber +g(\operatorname{grad}\Delta f,e_j)-g(\operatorname{grad}\Delta f,\operatorname{grad}f)g(e_j,\operatorname{grad}f)\\
   & & -(\Delta f)^2g(e_j,\operatorname{grad}f)
                -\frac{1}{2}(\Delta f)g(e_j,\operatorname{grad}\|\operatorname{grad}f\|^{2}).
\end{eqnarray}
The Theorem \ref{th3} follows from equation (\ref{eq3.2}) and the following
$$\operatorname{trace}_g(\nabla)^2\operatorname{grad}f=\operatorname{Ricci}(\operatorname{grad}f)+\operatorname{grad}\Delta f.$$
\end{proof}

\begin{corollary}\label{co3.1}
Let $(M,g)$ be a Riemannian manifold and $f\in C^\infty(M)$. Then
\begin{eqnarray*}
  (\operatorname{div}\chi)(\operatorname{grad}f)
   &=& \frac{1}{2}\|\operatorname{grad}\|\operatorname{grad}f\|^{2}\|^2
                +\|\operatorname{grad}f\|^{2}\operatorname{Ric}(\operatorname{grad}f,\operatorname{grad}f)\\
   & & +g(\operatorname{grad}\Delta f,\operatorname{grad}f)-(\Delta f)^2\|\operatorname{grad}f\|^2\\
   & & -\frac{1}{2}(\Delta f)g(\operatorname{grad}f,\operatorname{grad}\|\operatorname{grad}f\|^{2}).
\end{eqnarray*}
\end{corollary}

\begin{proposition}\label{pro3.1}
Let $(M,g)$ be a Riemannian manifold and $f\in C^\infty(M)$. We assume that
$\operatorname{Ric}\geq0$ and $\|\operatorname{grad}f\|$ is constant $<1$. Then
$$(\operatorname{div}\chi)(\operatorname{grad}f)\leq0.$$
\end{proposition}

\begin{proof}
By using the Bochner-Weitzenb\"{o}ck formula for smooth functions (see \cite{Sakai})
\begin{equation}\label{eq3.3}
\frac{1}{2}\Delta\|\operatorname{grad}f\|^{2}=\|\operatorname{Hess}_f\|^{2}+g(\operatorname{grad}f,\operatorname{grad}\Delta f)
+\operatorname{Ric}(\operatorname{grad}f,\operatorname{grad}f),
\end{equation}
and Corollary \ref{co3.1}, we find that
\begin{eqnarray}\label{eq3.4}
  (\operatorname{div}\chi)(\operatorname{grad}f)
   &=&\nonumber \frac{1}{2}\|\operatorname{grad}\|\operatorname{grad}f\|^{2}\|^2
                -(1-\|\operatorname{grad}f\|^{2})\operatorname{Ric}(\operatorname{grad}f,\operatorname{grad}f)\\
   & &\nonumber +\frac{1}{2}\Delta\|\operatorname{grad}f\|^{2}-\|\operatorname{Hess}_f\|^{2}
   -(\Delta f)^2\|\operatorname{grad}f\|^2\\
   & & -\frac{1}{2}(\Delta f)g(\operatorname{grad}f,\operatorname{grad}\|\operatorname{grad}f\|^{2}).
\end{eqnarray}
The Proposition \ref{pro3.1} follows from (\ref{eq3.4}) with $\|\operatorname{grad}f\|=C^{st}<1$.
\end{proof}

\begin{corollary}
Let $(M,g)$ be a Riemannian manifold and $f\in C^\infty(M)$. We assume that
$\operatorname{Ric}\geq0$ and $\|\operatorname{grad}f\|$ is constant $<1$. Then,
$(\operatorname{div}\chi)(\operatorname{grad}f)=0$ if and only if $f$ is an affine function on $(M,g)$.
Moreover, if $\operatorname{Ric}>0$, then $(\operatorname{div}\chi)(\operatorname{grad}f)=0$ if and only if $f$ is
a constant function on $M$.
\end{corollary}

\begin{remark}
If the Riemannian manifold $(M,g)$ has $\operatorname{Ric}\geq0$ and $\|\operatorname{grad}f\|$ is constant $<1$.
Then, 
\begin{itemize}
  \item $(\operatorname{div}\chi)(\operatorname{grad}f)=0$ implies that the identity maps $\tilde{I}_c$ and $\tilde{I}_d$ are harmonic;
  \item $(\operatorname{div}\chi)(\operatorname{grad}f)=0$ if and only if the identity map $\tilde{I}_c$ is harmonic.
\end{itemize}
 \end{remark}

\begin{corollary}
Let $(M,g)$ be a Riemannian manifold and $f\in C^\infty(M)$ such that $\|\operatorname{grad}f\|$ is constant $<1$. We assume that $\operatorname{grad}f$ is a conformal vector field,
i.e., $\operatorname{Hess}_f=\lambda g$ for some smooth function $\lambda$ on $M$. Then
\begin{eqnarray*}
  \operatorname{div}\chi(\cdot,\operatorname{grad}f)
   &=& -(1-\|\operatorname{grad}f\|^{2})\operatorname{Ric}(\operatorname{grad}f,\operatorname{grad}f) \\
   & & -\|\operatorname{Hess}_f\|^{2} +(1-\|\operatorname{grad}f\|^2 )(\Delta f)^2.
\end{eqnarray*}
\end{corollary}
\begin{proof}
The proof follows from equation (\ref{eq3.4}) and the formula (see \cite{Peter})
$$\operatorname{div}\chi(\cdot,\operatorname{grad}f)=(\operatorname{div}\chi)(\operatorname{grad}f)
+\frac{1}{2}\langle\mathcal{L}_{\operatorname{grad}f}g,\chi\rangle,$$
with $\mathcal{L}_{\operatorname{grad}f}g=2\operatorname{Hess}_f$ and $\Delta f=m\,\lambda$.
\end{proof}

\subsection*{Conflict of interest}
The author declares no conflict of interest.

\subsection*{Data availability}
Not applicable.


\begin{thebibliography}{99}


\bibitem{baird} P. Baird, J. C. Wood, {\it Harmonic morphisms between Riemannain manifolds}, Clarendon Press, Oxford (2003).

\bibitem{ba} A. Benkartab,  A. Mohammed Cherif, {\it New methods of construction for biharmonic maps}, Kyungpook Math. J. {\bf59}(2019), 135--147.

\bibitem{ba2} A. Benkartab, A. Mohammed Cherif, {\it Deformations of Metrics and Biharmonic Maps}, Communications in Mathematics, {\bf28}(2020), 263--275.

\bibitem{EL1}J. Eells, L. Lemaire,{\it A report on harmonic maps}, Bull. London Math. Soc. {\bf16} (1978) 1-68.

\bibitem{EL2} J. Eells, L. Lemaire, {\it Another report on harmonic maps}, Bull. London Math. Soc. {\bf20} (1988) 385-524.

\bibitem{eells} J. Eells and J. H. Sampson, {\it Harmonic mappings of Riemannian manifolds},  Amer. J. Math.
{\bf 86} (1964), 109--160.

\bibitem{MM23} B. Merdji, A. Mohammed Cherif, New types of metrics deformations and applications to $p$-biharmonic maps, Journal of the Indian Math. Soc., {\bf90}, (3-4) (2023), 387--400.


\bibitem{U1994} C. Udri\c{s}te, {\it Convex Functions and Optimization Methods on Riemannian Manifolds}, Mathematics and Its Applications, Kluwer Academic Publishers: Dordrecht, The Netherlands, 1994.

\bibitem{Sakai} T. Sakai, {\it Riemannian geometry}, Shokabo, Tokyo (1992). (in Japanese).

\bibitem{Peter} P. Topping,  {\it Lectures on the Ricci Flow}. Number 325 in London Mathematical Society Lecture Note Series. Cambridge University Press, October, 2006.

\end{thebibliography}
\end{document}